\documentclass[12pt]{amsart}
\usepackage{amsmath,amsthm,amssymb,fullpage,caption}
\usepackage{tikz}
\usepackage{pgfplots}
\usetikzlibrary{decorations.markings,decorations.pathmorphing}
\usetikzlibrary{calc, shapes, shapes.symbols, arrows, intersections, pgfplots.fillbetween}
\newcommand{\size}[1]{\left \vert #1 \right \vert}
\newcommand{\ceil}[1]{\left \lceil #1 \right \rceil}
\newcommand{\floor}[1]{\left \lfloor #1 \right \rfloor}
\newcommand{\cart}{\, \Box \,}

\newcommand{\burn}{c_b}
\newcommand{\capt}{\mathrm{capt}}
\newcommand{\captb}{\mathrm{capt}_b}
\newcommand{\dist}{\mathrm{dist}}

\newtheorem{theorem}{Theorem}[section]
\newtheorem{lemma}[theorem]{Lemma}
\newtheorem{conj}[theorem]{Conjecture}
\newtheorem{prop}[theorem]{Proposition}

\newtheorem{cor}[theorem]{Corollary}

\begin{document}

\title{Cops, robbers, and burning bridges}

\author{William B. Kinnersley}
\address{Department of Mathematics, University of Rhode Island, University of Rhode Island, Kingston, RI, USA, 02881}
\email{\tt billk@uri.edu}

\author{Eric Peterson}
\address{Department of Mathematics, University of Rhode Island, University of Rhode Island, Kingston, RI, USA, 02881}
\email{\tt epeterson11492@uri.edu}


\begin{abstract}
We consider a variant of Cops and Robbers wherein each edge traversed by the robber is deleted from the graph.  The focus is on determining the minimum number of cops needed to capture a robber on a graph $G$, called the {\em bridge-burning cop number} of $G$ and denoted $\burn(G)$.  We determine $\burn(G)$ exactly for several elementary classes of graphs and give a polynomial-time algorithm to compute $\burn(T)$ when $T$ is a tree.  We also study two-dimensional square grids and tori, as well as hypercubes, and we give bounds on the capture time of a graph (the minimum number of rounds needed for a single cop to capture a robber on $G$, provided that $\burn(G) = 1$). 
\end{abstract}

\maketitle

\begin{section}{Introduction}

The game of {\em Cops and Robbers} is a well-studied model of pursuit and evasion.  Cops and Robbers is played by two players: one controls a team of one or more {\em cops}, while the other controls a single {\em robber}.  The cops and robber all occupy vertices of a graph $G$ and take turns moving from vertex to vertex.  At the outset of the game, each cop chooses her initial position on $G$, after which the robber does the same.  (Multiple cops may occupy a single vertex simultaneously.)  Thereafter, the game proceeds in {\em rounds}, each consisting of a cop turn and a robber turn.  On the cop's turn, every cop may either remain in place or move to a neighboring vertex; on the robber's turn, he may do the same.  The cops win if some cop ever occupies the same vertex as the robber, at which time we say that cop {\em captures} the robber.  Conversely, the robber wins if he can perpetually avoid capture.  The cops and robber know each others' positions at all time.

Many variants of Cops and Robbers have been studied, each modeling pursuit and evasion in a slightly different context.  For example, the robber may move faster than the cops \cite{FGNS10,FKL12}, or the cops may have only partial information about the robber's location~\cite{CCDDFM17,DDTY13}, or the two players may have different sets of edges available to them~\cite{NN93}.  In these variants, one typically seeks to determine the minimum number of cops needed to capture a robber on a graph $G$.  In the usual model of Cops and Robbers, this quantity is deemed the {\em cop number} of $G$ and denoted $c(G)$.  For more background on Cops and Robbers, we refer the reader to~\cite{BN11}.

In this paper, we introduce and study a variant of Cops and Robbers wherein the robber, after traversing an edge, deletes that edge from the graph.  For example, perhaps the edges of our graph represent bridges joining various regions, and the robber burns each bridge as he passes over it, denying its future use both to the cops and to the robber.  (We require that the robber always burns every edge he uses; he may not elect to leave an edge intact.)  Aside from this change, the rules are the same as in the usual model of Cops and Robbers.  We refer to this game as {\em bridge-burning Cops and Robbers} and define the {\em bridge-burning cop number}, denoted $\burn(G)$, to be the minimum number of cops needed to capture a robber on $G$ in this model.

In general, $c(G)$ and $\burn(G)$ are not directly comparable, and the relationship between the two can be surprising.  As the bridge-burning game wears on, the robber deletes more and more edges from $G$ and thus has fewer escape routes.  Hence one might expect that generally $\burn(G) \le c(G)$, and indeed, sometimes this is the case.  However, in the usual model of Cops and Robbers, the robber can only play defensively, while in the bridge-burning game, he can adopt an offensive tack: if the robber can disconnect the graph and leave himself in a different component from all of the cops, then he wins.  Thus sometimes $\burn(G) > c(G)$, since there must be enough cops to capture the robber before he can pull off this feat.

In this paper, we investigate the bridge-burning game on a variety of graph classes on which the usual model of Cops and Robbers is well-understood.  In Section~\ref{sec:simple}, we determine the bridge-burning cop numbers of paths, cycles, and complete graphs.  We also generalize the elementary bound $c(G) \le \gamma(G)$ by giving two upper bounds on $\burn(G)$ in terms of domination-like parameters of $G$.  In Section~\ref{sec:trees}, we give a polynomial-time algorithm to compute $\burn(T)$ when $T$ is a tree.  In Section~\ref{sec:grids}, we examine square grids and tori.  Theorem~\ref{thm:2byn} states that when $G$ is a $2 \times n$ grid, $\burn(G) = \ceil{\frac{n+2}{9}}$, while Theorems~\ref{thm:torus} and \ref{thm:grid} state that when $G$ is an $m \times n$ square grid or torus, $\frac{mn}{121} \le \burn(G) \le (1+o(1))\frac{mn}{112}$.  We also show in Theorem~\ref{thm:hypercube} that the bridge-burning cop number of the $n$-dimensional hypercube, $Q_n$, is always 1.  Finally, in Section~\ref{sec:capt}, we briefly consider the concept of {\em capture time} -- the number of rounds needed for a single cop to win on a graph with bridge-building cop number 1.  Theorem~\ref{thm:capt_upper} shows that among all $n$-vertex graphs $G$, the capture time of $G$ is $O(n^3)$, while Theorem~\ref{thm:capt_lower} shows that there exist $n$-vertex graphs having capture time $\Omega(n^2)$.  Finally, in Section~\ref{sec:open}, we suggest some directions for future research.
\end{section}

% general stuff
\begin{section}{General Bounds}\label{sec:simple}

We begin with elementary observations about the bridge-burning game, starting with the value of $\burn(G)$ on several elementary classes of graphs.

\begin{prop}\label{prop:simple}
\mbox{}
\begin{itemize}
\item [(a)] $\displaystyle \burn(K_n) = 1$ for all $n$.
\medskip
\item [(b)] $\displaystyle \burn(C_n) = 1$ for $n \ge 3$.
%\smallskip
\item [(c)] $\displaystyle \burn(P_n) = \left \{\begin{array}{ll}1, \,\, &\text{if }n \le 5\\2, &\text{otherwise}\end{array} \right .$
\end{itemize}
\end{prop}
\begin{proof}
\mbox{}
\begin{itemize}
\item [(a)] To capture a robber on $K_n$, the cop starts on an arbitrary vertex; no matter where the robber starts, the cop can capture him on her first turn.\\

\item [(b)] On $C_n$, the cop starts on an arbitrary vertex and simply moves closer to the robber on each turn.  Once the robber has taken his first step (and hence burnt the corresponding edge), he finds himself at one endpoint of a path.  Subsequent moves by the robber only serve to shorten this path, so eventually the cop will reach him.\\

\item [(c)] For $n = 2$, the claim is trivial.  

For $3 \le n \le 5$, label the vertices of the path $v_1, v_2, \dots, v_n$ in order.  The cop begins on $v_3$.  If the robber begins the game on or adjacent to the cop, then the cop wins on her first turn.  Otherwise, it must be that the robber begins at one endpoint of the path.  On her first turn, the cop moves toward the robber; she is now adjacent to the robber.  The robber cannot leave his current vertex, since that would result in immediate capture; however, if he remains in place, then the cop captures him on her next turn.  Either way, the cop wins.

For $n \ge 6$, we claim that two cops are necessary and sufficient to capture the robber.  It is clear that two cops can capture the robber: one begins at each endpoint, and on each turn they both move closer to the robber.  To see that two cops are necessary, we give a strategy for the robber to avoid capture by a single cop.  Label the vertices of the path $v_1, v_2, \dots, v_n$ in order.  Since $n \ge 6$, at least one of $v_2$ and $v_{n-1}$ must not be adjacent to the cop's initial position; by symmetry suppose this is true of $v_2$.  The robber begins the game on $v_2$.  By assumption the cop cannot capture him on her first turn.  On the robber's first turn, he moves to $v_1$, thereby burning edge $v_1v_2$.  Now $v_1$ is isolated, so the cop can never reach the robber. 
\end{itemize}
\end{proof}

Note that even on such elementary graphs, the cop number and the bridge-burning cop number can differ: for $n \ge 3$ we have $c(C_n) = 2$ but $\burn(C_n) = 1$, and for $n \ge 6$ we have $c(P_n) = 1$ but $\burn(P_n) = 2$.  We will see later (in Theorem \ref{thm:trees}) that the difference between $c(G)$ and $\burn(G)$ can be arbitrarily large.

The argument in the proof of Proposition \ref{prop:simple}(c) suggests a natural heuristic strategy for the robber: attempt to move in such a way that he ends up in a different component from every cop.  Certainly if the robber accomplishes this, then he wins.  However, this is not the only way for the robber to win.  For example, in the graph shown in Figure \ref{fig:stalemate}, the robber can evade a single cop by causing a stalemate.  If the cop begins on $v$ or $y$, then the robber can safely begin on $x$; on his first turn the robber moves to $z$ and wins.  Likewise, if the cop begins on $x$ or $z$, then the cop begins on $v$ and subsequently moves to $y$.  Thus the cop must begin on $u$ or $w$; suppose without loss of generality that she begins on $u$.  The robber now begins on $w$.  If the cop moves to $v$, then the robber can move to $x$ and subsequently to $z$, thereby winning the game.  Likewise, if the cop moves to $x$, then the robber can move to $v$ and from there to $y$.  Thus the cop's only reasonable option is to remain at $u$; the robber responds by remaining at $w$.  The robber wins if the cop ever leaves $u$, so the cop must remain at $u$ perpetually and thus cannot capture the robber.

\begin{figure}[h]
\begin{center}
\begin{tikzpicture}
[inner sep=0mm, thick,
 vertex/.style={draw=black, fill=black, circle, minimum size=0.2cm},
 dot/.style={draw=black,fill=black, circle, minimum size=0.04cm},
 vertexlabel/.style={draw=none, fill=none, shape=rectangle, inner sep=2pt, font=\small},
 textnode/.style={draw=none, fill=none, shape=rectangle, inner sep=2pt},
 xscale=1.0,yscale=1.0]

\node (u) at (0,1) [vertex] {};
\node at ($(u)+(0,0.12)$) [vertexlabel, anchor=south] {$u$}; 

\node (v) at (-1,0) [vertex] {};
\node at ($(v)+(-0.06,0.08)$) [vertexlabel, anchor=south east] {$v$};

\node (w) at (0,-1) [vertex] {};
\node at ($(w)+(0,-0.12)$) [vertexlabel, anchor=north] {$w$};

\node (x) at (1,0) [vertex] {};
\node at ($(x)+(0.06,0.08)$) [vertexlabel, anchor=south west] {$x$};

\node (y) at (-2,0) [vertex] {};
\node at ($(y)+(0,0.12)$) [vertexlabel, anchor=south] {$y$};

\node (z) at (2,0) [vertex] {};
\node at ($(z)+(0,0.12)$) [vertexlabel, anchor=south] {$z$};

\draw (u) -- (v) -- (w) -- (x) -- (u);
\draw (v) -- (y);
\draw (x) -- (z);
\end{tikzpicture}
\end{center}
\caption{}
\label{fig:stalemate}
\end{figure}

It is well-known that in the usual model of Cops and Robbers, we have $c(G) \le \gamma(G)$, where $\gamma(G)$ denotes the minimum size of a dominating set -- that is, a set $S \subseteq V(G)$ such that every vertex in $G$ either belongs to $S$ or has a neighbor in $S$.  It is clear that $\burn(G) \le \gamma(G)$, since placing one cop on each vertex of a dominating set allows the cops to win on their first turn.  However, in the context of the bridge-burning game, we can strengthen this bound, and in fact we do so in two different ways.

\begin{theorem}\label{thm:dominating_cliques} 
If there exist cliques $S_1, S_2, \dots, S_k$ in $G$ such that $S_1 \cup S_2 \cup \dots \cup S_k$ is a dominating set of $G$, then $\burn(G) \le k$.
\end{theorem}
\begin{proof}
Let $S_1, S_2, \dots , S_k$ be cliques whose union dominates $G$, and let $S = S_1 \cup S_2 \cup \dots S_k$.  We show how $k$ cops can capture a robber on $G$.  Label the cops $c_1,...,c_k$, and let each cop $c_i$ begin the game on any vertex in $S_i$; throughout the game, she will remain in $S_i$.  For $i \in \{1, \dots, k\}$, the robber cannot start on any vertex in $S_i$ without being captured immediately by $c_i$.  Instead, the robber must start on some vertex not in $S$.  The cops now play as follows.

Suppose the robber currently occupies vertex $v$.  Since $S$ dominates $G$, vertex $v$ must be adjacent to some vertex in $S$, say $u$; suppose $u \in S_i$.  Cop $c_i$ moves to $u$, while every other cop remains on her current vertex.  The robber cannot remain on $v$ without being captured, and he cannot move to a vertex in $S$, so on his next turn, he must flee to another vertex not in $S$.  Since $G$ has only finitely many edges, the robber cannot flee forever; he will eventually be captured.
\end{proof}

\begin{cor}\label{cor:complete_bipartite} 
For all $m$ and $n$, we have $\burn(K_{m,n}) = 1$.
\end{cor}
\begin{proof}
In $K_{m,n}$, any pair of adjacent vertices forms a dominating set; the result now follows from Theorem \ref{thm:dominating_cliques}.
\end{proof}

We can also bound $\burn(G)$ by considering {\em distance-2 domination}.  A {\em distance-2 dominating set} in $G$ is a set $S$ of vertices such that each vertex in $G$ is distance at most 2 from some member of $S$; the {\em distance-2 domination number} of $G$, denoted $\gamma_2(G)$, is the minimum size of a distance-2 dominating set.  Note that $\gamma_2(G) \le \gamma(G)$ for every graph $G$.

\begin{theorem}\label{thm:dist2_dominating}
For every connected graph $G$, we have $\burn(G) \le \gamma_2(G)+1$.
\end{theorem}
\begin{proof}
Let $\{v_1, v_2, \dots, v_k\}$ be a distance-2 dominating set of $G$.  We give a strategy for $k+1$ cops $c_1, c_2, \dots, c_k, c^*$ to capture a robber on $G$.  For $1 \le i \le k$, cop $c_i$ begins the game on vertex $v_i$, which we refer to as $c_i$'s {\em home vertex}; cop $c^*$ begins at an arbitrary vertex.  Each vertex of $G$ is initially within distance 2 of some cop's home vertex.  The cops will maintain this property throughout the game by ensuring that the robber can never safely reach a vertex adjacent to any home vertex (and hence cannot delete any edges incident to such a vertex).  

During the game, the cops move as follows.  If on any cop turn some cop is adjacent to the robber, then she captures the robber.  Otherwise, the cops $c_1, c_2, \dots, c_k$ will ensure that $G$ remains connected at all times and that the robber cannot occupy any neighbor of a cop's home vertex (without being captured on the cops' subsequent turn).  Moreover, the cops will move so that each $c_i$ is always on or adjacent to her home.  On each turn, $c^*$ moves arbitrarily toward the robber; this will always be possible since $G$ will remain connected. 

Cops $c_1, c_2, \dots, c_k$ move in response to the robber's last move, with the aim of keeping $G$ connected.  Suppose the robber moves to some vertex $v$.  If $v$ is not incident to a cut-edge, then each $c_i$ either returns home or, if already home, remains home.  Otherwise, consider a cut-edge $uv$ incident to $v$.  By choice of the home vertices, there is some $v_i$ within distance 2 of $u$.  Every cop $c_j$ for $j \not = i$ either returns home or stays home.  If $c_i$ herself is currently home, then she moves to some neighbor of $u$, thereby preventing the robber from moving to $u$ and disconnecting the graph.  We claim that in fact $c_i$ must have been home and, thus, can defend $u$.  Suppose for the sake of contradiction that $c_i$ is not currently home.  This implies that on the preceding cop turn, she moved to defend a different cut-edge.  In particular, suppose that on the last cop turn, the robber occupied vertex $v'$, which was incident to cut-edge $u'v'$, and $c_i$ moved to a neighbor of $u'$.  Since the robber moved from $v'$ to $v$ on his last turn, the vertices $v'$ and $v$ were adjacent before the robber's move.  Moreover, $u'$ and $u$ were connected by a walk through $v_i$.  Thus when the robber occupied $v'$, the edges $u'v'$ and $uv$ shared a common cycle, so $u'v'$ must not have been a cut-edge, a contradiction.  Thus, the cops can successfully defend any given cut-edge.

We must also show that no cop can be forced to defend more than one cut-edge on any given turn and that the robber never reaches any neighbor of a home vertex (without being captured on the cops' ensuing turn).  Suppose the robber occupies some vertex $v$ at distance 2 from $v_i$.  If $v_i$ and the robber's vertex have two common neighbors $t$ and $u$, then neither $vt$ nor $vu$ can be a cut-edge; consequently, no cop can be forced to defend more than one cut-edge on a single turn.  Additionally, if $v$ and $v_i$ have two or more common neighbors, then $c_i$ must occupy $v_i$, since on her last turn she either moved there or remained there.  If instead $v$ and $v_i$ have only one common neighbor $w$, then $c_i$ either occupies $v_i$ or $w$, depending on whether $vw$ is a cut-edge.  In any case, $c_i$ clearly prevents the robber from reaching any neighbor of $v_i$. 

We claim that by playing in this manner, the cops eventually capture the robber.  The cops' strategy prevents the robber from visiting a vertex adjacent to any $v_i$, so throughout the game, every vertex of $G$ remains within distance 2 of some $v_i$.  Moreover, the $c_i$ prevent the robber from disconnecting the graph, so the graph remains connected, hence $c^*$ can execute her part of the strategy.  Finally, $c^*$'s movement ensures that the robber cannot remain still indefinitely.  Since $G$ is finite, eventually the robber will run out of safe moves and thus will be captured.
\end{proof}

In the context of Theorem~\ref{thm:dominating_cliques}, taking one vertex of every $S_i$ yields a distance-2 dominating set of size $k$.  If in fact $\gamma_2(G) = k$, then Theorem~\ref{thm:dominating_cliques} yields $\burn(G) \le k$, while Theorem~\ref{thm:dist2_dominating} yields $\burn(G) \le k+1$.  Thus Theorem~\ref{thm:dominating_cliques} may be stronger than Theorem~\ref{thm:dist2_dominating}.  However, in general this approach will not yield a minimum distance-2 dominating set, and so typically Theorem~\ref{thm:dist2_dominating} is stronger than Theorem~\ref{thm:dominating_cliques}.  We also note that Theorem~\ref{thm:dist2_dominating} is tight: the graph in Figure~\ref{fig:stalemate} has distance-2 domination number 1 but bridge-burning cop number 2.\\
\end{section}

\begin{section}{Trees}\label{sec:trees}

In this section, we study the bridge-burning game on trees.  In the usual model of Cops and Robbers, trees are easy to analyze: it is well-known that for every tree $T$, we have $c(T) = 1$.  In the bridge-burning model, things are more complicated; in fact, there exist trees with arbitrarily large bridge-burning cop number.  Below, we give a polynomial-time algorithm to determine the bridge-burning cop number of a tree.  The key idea underlying the algorithm is the same as that behind Proposition \ref{prop:simple}(c): if the robber can safely start on some vertex adjacent to a leaf, then on his next turn he can isolate himself on the leaf and thus win.  To prevent this, the cops must ensure that after their initial placement, each leaf is within distance 2 of at least one cop.

We say that a leaf of $v$ of a tree is {\em guarded} if some cop begins the game within distance 2 of $v$ and {\em unguarded} otherwise.  A cop within distance 2 of $v$ is said to {\em guard} $v$.

\begin{theorem}\label{thm:trees}
Let $T$ be a tree.  Consider the following algorithm:
\begin{enumerate}
\item Choose an arbitrary root $r$ for $T$.

\item Out of all unguarded leaves of $T$, let $v$ be one furthest from the root.

\item If $v = r$ or $v \in N(r)$, place a cop at $r$; otherwise, place a cop at the grandparent of $v$.\looseness=-1

\item Repeat steps 2 and 3 until all leaves of $T$ have been guarded. 
\end{enumerate}
If $N$ denotes the number of cops placed by the algorithm, then $\burn(T) = N$.  Moreover, this algorithm can be executed in polynomial time.
\end{theorem} 
\begin{proof}
It is clear that the algorithm can be executed in polynomial time, so we need only show that $\burn(T) = N$.  We first show that $\burn(T) \geq N$, i.e. that we need at least $N$ cops to capture a robber on $T$.  If some leaf $v$ of $T$ is unguarded, then the robber can begin the game on some neighbor of $v$ and, on his first turn, move to $v$.  This puts the robber on the isolated vertex $v$, so the cops can never capture him.  Hence, the cops must ensure that after their initial placement, every leaf of $T$ is guarded.  We claim that this requires at least $N$ cops.

Consider the leaf $v$ chosen in the first iteration of step 2 of the algorithm.  If $v=r$ or $v \in N(r)$, then all unguarded leaves are within distance 1 of $r$, so placing a cop on $r$ guards all leaves.  Suppose instead that $v$ is at least distance 2 from $r$.  Let $u$ and $t$ be the parent and grandparent of $v$, respectively.  The cops must ensure that some cop guards $v$, which requires placing a cop at $t$, $u$, or some child of $u$.  By choice of $v$, no child of $u$ has any unguarded leaves as descendants.  Consequently, any unguarded leaf that would be guarded by a cop at $u$ or some child of $u$ would also be guarded by a cop at $t$.  Thus, some optimal cop placement (i.e. one that guards all leaves using the fewest possible cops) places a cop at $t$, just as the algorithm does.  Repeating this argument, we see that each cop placed by the algorithm is placed optimally with respect to guarding the chosen unguarded leaf, and hence the algorithm produces an optimal cop placement.  This completes the proof that $\burn(T) \ge N$.

To show that $\burn(G) \leq N$, we argue that the cops can always capture a robber starting from the initial cop placement produced by the algorithm.  Since every leaf is guarded, if the robber starts on a neighbor of a leaf, then the cops can capture him on their first turn.  If instead he starts on a leaf, then some cop can move to the neighbor of that leaf, thereby trapping him, and the cops can capture him on their next turn.  

Suppose the robber starts on any other vertex in $T$.  Each turn, every cop moves one step closer to the robber (if possible).  Initially, for every leaf in $T$, the unique path between the robber and that leaf contains at least one cop.  We claim that after every robber turn, the path between the robber and any leaf in the same component of $T$ contains a cop.  This property is maintained by the cops' strategy, so we need only consider what happens on the robber's turns.  Suppose that it is the robber's turn and that the property holds.  If the robber remains on his current vertex, then the property still holds.  Otherwise, the robber's move either takes him toward or away from any given leaf; in the former case the path from the robber to that leaf still contains a cop, and in the latter case the robber and leaf are now in different components.  Since there must always be a leaf of $T$ in the robber's component,  there must always be a cop in the robber's component, and hence the cops eventually win.
\end{proof}
\end{section}

% grids and hypercubes
\begin{section}{Grids and Hypercubes}\label{sec:grids}

In this section, we investigate the bridge-burning game played on two-dimensional square grids and tori.  As with trees, these are graphs on which the bridge-burning model is much more difficult to analyze than the standard model.  It is known that every two-dimensional grid has cop number at most 2 (see \cite{MM87}) and every two-dimensional torus has cop number at most 3 (see \cite{LP17}), but a
s we will show, there exist grids and tori having arbitrarily large bridge-burning cop number.

The {\em $m \times n$ square grid}, which we denote $G_{m,n}$, is the Cartesian product of the paths $P_n$ and $P_m$; the {\em $m \times n$ square torus}, denoted $T_{m,n}$, is the Cartesian product of $C_n$ and $C_m$.  We view the vertex sets of both graphs as the set of ordered pairs $(i,j)$ with $0 \le i \le n-1$ and $0 \le j \le m-1$.  For fixed $i$, we say that vertices of the form $(i,k)$ are in {\em column $i$} of the grid; similarly, those of the form $(\ell,i)$ are in {\em row $i$}.  Note that both $G_{m,n}$ and $T_{m,n}$ have $m$ rows and $n$ columns; a vertex's column is indexed by its first coordinate, while its row is indexed by the second coordinate.

We say that vertex $(i,j)$ is {\em to the left} of $(i',j')$ if $i < i'$ and {\em to the right} if $i > i'$.  Similarly, $(i,j)$ is {\em above} $(i',j')$ if $j < j'$ and {\em below} if $j > j'$.  When a player moves from vertex $(i,j)$ to $(i+1,j)$, we say they {\em move right}; likewise, when they move to $(i-1,j)$, $(i,j+1)$, or $(i,j-1)$ we say that they {\em move right}, {\em move down}, or {\em move up}, respectively.

We begin with $2 \times n$ grids.  In this setting, a robber who starts near the left or right ends of the grid has somewhat more power than one who starts in the middle: the ends of the grid contain vertices of low degree, which makes it easier for the robber to isolate himself.  The cops can prevent this by stationing cops ``close enough'' to the ends of the grid; the following lemma formalizes this idea.

\begin{lemma}\label{lem:2byn_robber_outside} 
Consider the game played on $G_{2,n}$.  If a cop starts in column $j$ where $1 \leq j \leq 3$ and the robber starts to the left of the cop, then the cop can capture the robber.  Similarly, if a cop starts in column $k$ where $n-4 \leq k \leq n-2$ and the robber starts to the right of the cop, then the cop can capture the robber.
\end{lemma}
\begin{proof} 
First, we suppose the cop starts on vertex $(3,0)$ and the robber starts to the left of the cop; a symmetric argument suffices for cops starting on vertices $(3,1)$, $(n-4,0)$, or $(n-4,1)$ with the robber to the left, right, or right, respectively.  On her first turn, the cop always moves left to $(2,0)$.  Henceforth, the cop plays as explained in the cases below.

\begin{itemize}
\item {\bf Case 1:} the robber starts on $(2,0)$.  The cop captures him immediately on her first turn.

\medskip
\item {\bf Case 2:} the robber starts on $(1,0)$.  If the robber remains on $(1,0)$ on his first turn, then the cop captures him on her next turn.  If the robber moves left to $(0,0)$, the cop moves left to $(1,0)$.  Regardless of the robber's next move, the cop then moves down to $(1,1)$ and the robber is trapped, so the cop wins. If instead the robber moves down to $(1,1)$ on his first turn, then the cop moves down to $(2,1)$.  The robber must move left to $(0,1)$ to avoid capture.  The cop moves back up to $(2,0)$ and traps the robber, ensuring her win.

\medskip
\item {\bf Case 3:} the robber starts on $(0,0)$.  Regardless of the robber's first move, the cop moves left to $(1,0)$ on her second turn.  If the robber moves right to $(1,0)$ on his first turn, then the cop captures him on her subsequent turn.  If the robber instead chooses to stay on $(0,0)$ on his first turn, then on his second turn, he must move down to $(0,1)$ to avoid capture.  The cop can now move down to $(1,1)$ and trap the robber.  If the robber moves down to $(0,1)$ on his first turn, then he must remain on $(0,1)$ or move right to $(1,1)$ on his second turn; in either case, the cop moves down to $(1,1)$ and either traps or captures the robber.

\medskip
\item {\bf Case 4:} the robber starts on $(2,1)$.  Since the cop moves to $(2,0)$ on her first turn, the robber now must move either left or right to avoid capture on the cop's ensuing turn.  For the remainder of the game, on each turn, the cop moves horizontally into the same column as the robber.  Consequently, on the ensuing robber turn, the robber must continue moving horizontally in the same direction to avoid capture.  Since the graph is finite, the robber cannot keep this up forever, so eventually the cop wins.

\medskip
\item {\bf Case 5:} the robber starts on $(1,1)$.  If the robber moves right to $(2,1)$, then the cop captures him on her next turn.  If the robber remains on $(1,1)$, then the cop moves left to $(1,0)$.  The robber must now move either horizontally on his next turn to avoid capture; the cop can now capture him using the strategy described in Case 4.  If the robber moves left to $(0,1)$, then the cop moves left to $(1,0)$ and traps the robber.  In any case, the cop wins.

\medskip
\item {\bf Case 6:} the robber starts on $(0,1)$.  Regardless of the robber's first move, the cop moves left to $(1,0)$ on her second turn.  If the robber moves right to $(1,1)$ on his first turn, then on his next turn, he must move left or right to avoid capture; once again, the cop can now capture him using the strategy given in Case 4.  If the robber moves up to $(0,0)$, then he must remain on $(0,0)$ on his next turn, after which the cop moves left to capture him.  If the robber remains on $(0,1)$ after his first move, then after the cop's second move, the robber must again remain on $(0,1)$ to avoid capture.  The cop now moves down to $(1,1)$, forcing the robber to move up to $(0,0)$; the cop moves up to $(0,1)$ and traps the robber.
\end{itemize}

\medskip
This establishes the claim for the case where the cop starts in columns 3 or $n-4$; similar arguments suffice if the cop begins in columns 1 or 2 (or $n-2$ or $n-3$).
\end{proof}

We next consider how to deal with a robber who begins in the middle of the grid, far from the edges.  The cops must be sure not to leave too large of a ``gap'' between cops, lest they give the robber enough freedom to isolate himself.

\begin{lemma}\label{lem:2byn_middle} 
Consider the game on $G_{2,n}$, and suppose two cops start in the same row at a distance of $k$ columns apart where $k \leq 9$.  
If the robber starts between them, then the cops can capture him.
\end{lemma}
\begin{proof} 
By giving a cop strategy, we show that regardless of the robber's strategy, he cannot win unless the separation between cops is at least 10 columns.  The general strategy for the cops will be to move horizontally toward the robber; we give a full specification below.  First, we claim that the robber must make at least two vertical moves in order to win.

Suppose the robber has a winning strategy using fewer than two vertical moves.  Note that if the robber ever moves to a vertex directly above or below a cop with the corresponding vertical edge intact, then that cop captures him immediately.  If the robber only moves horizontally, then he is eventually captured by one of the cops he starts between.  If the robber makes only one vertical move, then he enters a row with all of its edges intact, and as before, the cops can now trap the robber from both sides.  Thus, the robber cannot win without making at least two vertical moves.  

We now detail the cops' strategy.  Let $c_\ell$ (respectively, $c_r$) denote the cop that starts to the left (respectively, the right) of the robber.  Without loss of generality, assume both cops start in row 0.  In most circumstances, $c_\ell$ and $c_r$ both move horizontally inward towards the robber on every turn, and move vertically only when the robber is directly below them.  Exceptions to this are as follows:

\begin{itemize}
\item If moving horizontally would cause a cop to enter a vertex with no vertical edge, then the cop instead moves vertically to the other row and continues moving horizontally towards the robber for as long as possible.

\medskip
\item If the robber starts in row 0 and, on his first turn, moves horizontally away from one of the cops, then that cop subsequently moves down to row 1 and henceforth continues moving horizontally toward.

\medskip
\item If the robber starts in row 0 and his first four moves are to cycle back to his starting vertex (meaning that he has moved up, down, left, and right in some order), then only one cop remains in the robber's component.  From this point onward, this cop plays as in Lemma \ref{lem:2byn_robber_outside}, supposing that the she and the robber have each taken a single turn, and the robber's first move was vertical.  (We will argue below that by the time the robber returns to his starting vertex, the remaining cop is no more than three columns away, so she may indeed employ the strategy in Lemma~\ref{lem:2byn_robber_outside}.)

\medskip
\item If the robber starts in row 1 and moves up on his first turn, then both cops move down in response and, henceforth, move horizontally inward towards the robber on each turn.
\end{itemize}

\medskip
We claim that when the cops employ this strategy, the robber can avoid capture only if the cops start at least 10 columns apart.  We consider several cases.  We may clearly assume that the robber moves to an adjacent vertex on his first turn, since remaining in place only allows the cops to move closer to each other (and to the robber).

\medskip
{\bf Case 1:} the robber starts in row 0, initially moves horizontally $k_1$ times ($k_1 \geq 0)$, and then moves down.  Without loss of generality, assume that the robber initially moves right. 

\begin{itemize}
\item[(a)] Suppose $k_1 = 0$ (so the robber moves down on his first turn).  If the robber remains on this vertex for the remainder of the game, then he clearly loses, so suppose without loss of generality that he eventually moves right.  The robber must eventually return to row 0 in order to win, and he can do so no earlier than his third turn.  During this time, the cops have moved horizontally inward on each turn.  Thus, before returning to row 0, the robber must be at least 5 columns from each cop's starting position to avoid capture by that cop; the claim now follows.

\medskip
\item[(b)] Next suppose $k_1 \ge 1$, and suppose that after moving down, the robber moves right $k_2$ times before moving up (where $k_2 \ge 1$).  In total, the robber has returned to row 0 after at least $k_1+k_2+2$ turns and in doing so has moved $k_1+k_2$ columns to the right of his starting position.  To avoid capture by $c_r$, who has moved $k_1+k_2+3$ columns to the left during this time, the robber currently must be at least $k_1+k_2+4$ columns away from $c_r$'s starting position; thus he must have started the game at least $2k_1+2k_2+4$ columns to the left of $c_r$.  Additionally, to avoid capture by $c_{\ell}$, he must have started at least two columns to her right.  Consequently, for the robber to avoid capture, the cops must have started at least $2k_1+2k_2+6$ columns apart, which is at least 10, as claimed.

\medskip
\item[(c)] Finally suppose that $k_1 \ge 1$ and that after moving down, the robber moves left $k_2$ times before moving up (where $k_2 \ge 1$).  Since the robber initially moved right, $c_{\ell}$ has moved right once, down once, and right another $k_1+k_2$ times during her first $k_1 + k_2 + 2$ turns.  Thus, just before moving up, the robber must be at least $k_1+k_2+2$ columns to the right of $c_{\ell}$'s starting position to avoid capture by $c_{\ell}$; since the robber is now $k_2-k_1$ columns to the left of his starting position, he must have started at least $2k_1+2$ columns to the right of $c_{\ell}$.  Additionally, to avoid capture by $c_r$ (who moves left on each of his first $k_1+1$ turns), the robber must have started at least $2k_1+2$ columns to the left of $c_r$.   In total, for the robber to avoid capture, the cops must start at least $2k_1+2k_2+4$ columns apart.  

The claim now follows unless $k_1 = k_2 = 1$.  In this case, the robber moves right, down, left, and up in his first four moves, thereby returning to his starting vertex.  During this time, $c_\ell$ has responded by moving right once, down once, and right three times, moving a total of 4 columns closer to the robber.  Thus, if $c_\ell$ starts no more than 6 columns to the left of the robber, then $c_{\ell}$ will be able to capture the him using the strategy outlined in Lemma \ref{lem:2byn_robber_outside}.  Since the robber must also start at least 4 columns to the left of $c_r$ (as mentioned above), the result follows.
\end{itemize}

\medskip
{\bf Case 2:} the robber starts in row 1, moves horizontally $k_1$ times ($k_1 \geq 0$), and then moves up. Without loss of generality, assume that the robber initially moves right. % Note that in his first $k_1+1$ turns the robber has moved right $k_1$ times and $c_r$ has moved left $k_1+2$ times, so the robber must start at least $k_1+3$ columns to the left of $c_r$ to avoid capture.

\begin{itemize}
\item[(a)] Suppose $k_1 = 0$ (so the robber moves up on his first turn).  The robber must eventually return to row 1 in order to win and can do so no sooner than his third turn.  During this time, each cop moves down to row 1 and at least three columns closer to the robber.  At this point, all horizontal edges in row 1 are intact.  Thus, if either cop is now within two columns of the robber, then the cops will trap him regardless of his next move.  The claim now follows.

\medskip
\item[(b)] Next suppose that $k_1 \ge 1$ and that after moving up, the robber next moves right $k_2$ times (where $k_2 \geq 1$) before moving down.  Note that just before the robber moves down, he has moved $k_1+k_2$ columns to the right, while $c_r$ has moved $k_1+k_2+2$ columns to the left.  Thus the robber must have started at least $2k_1+2k_2+3$ columns to the left of $c_r$ to avoid capture by $c_r$.  Additionally, the robber must have started at least three columns to the right of $c_{\ell}$, since otherwise $c_{\ell}$ will capture him immediately after he moves up.  In total, for the robber to avoid capture, the cops must start at least $2k_1+2k_2+6$ columns apart; this is at least 10, as claimed.

\medskip
\item[(c)] Finally, suppose that $k_1 \ge 1$ and that after moving up, the robber next moves left $k_2$ times (where $k_2 \geq 1$) before moving down.  During his first $k_1+1$ turns, the robber has moved $k_1$ columns to the right; during her first $k_1+2$ turns, $c_r$ has moved $k_1+2$ columns to the left.  Thus for the robber to avoid capture by $c_r$, he must start at least $2k_1+3$ columns to her left.  Additionally, just before the robber moves down, he has moved $k_2-k_1$ columns to the left, while $c_{\ell}$ has moved $k_1+k_2+2$ columns to the right.  Thus the robber must start at least $2k_2+3$ columns to the right of $c_{\ell}$.  Once again, for the robber to avoid capture, the cops must start at least $2k_1+2k_2+6$ columns apart, and the claim follows.
\end{itemize}
\end{proof}

We are finally ready to determine $\burn(G_{2,n})$.

\begin{theorem}\label{thm:2byn} 
$\displaystyle \burn(G_{2,n}) = \left \lceil \frac{n+2}{9} \right\rceil$.
\end{theorem}
\begin{proof} 
We first show that $\burn(G) \le \ceil{\frac{n+2}{9}}$ by explaining how this many cops can capture the robber.  If $n \leq 7$, one cop suffices: the cop begins in column 3 (or in column $n-1$ if $n \le 3$) and, by Lemma \ref{lem:2byn_robber_outside}, has a winning strategy.  Otherwise, place one cop in column 3 and one cop in column $n-4$.  Next, starting from the cop in column 3, repeatedly place a cop 9 columns to the right of the previous cop until there are at most 9 columns between the cop just placed and the cop in column $n-4$.  In total, we have placed $\left \lceil \frac{n-7}{9} \right\rceil$ cops in the first $n-4$ columns in addition to the cop in column $n-4$, for a total of $\left \lceil \frac{n-7}{9} \right\rceil+1$ cops, which simplifies to $\ceil{\frac{n+2}{9}}$.  By Lemmas~ \ref{lem:2byn_robber_outside} and \ref{lem:2byn_middle}, this cop placement ensures that the cops can capture the robber, hence $\burn(G) \leq \left \lceil \frac{n+2}{9} \right\rceil$.

For the reverse inequality, suppose $k$ cops suffice to win the game.  Label the cops $c_1,...,c_k$ and suppose that cop $c_i$ begins in column $m_i$, with $m_i \le m_j$ whenever $i < j$.  

We first claim that if $m_1 \not = m_2$, then $m_1 \leq 3$.  Suppose that $4 \le m_1 < m_2$ and, without loss of generality, that $c_1$ starts on vertex $(m_1,0)$.  If the robber starts on $(0,0)$, then he can isolate himself on $(0,0)$ by moving right to $(1,0)$, down to $(1,1)$, left to $(0,1)$, and finally up to $(0,0)$; it is straightforward to verify that $c_1$ cannot reach the robber quickly enough to capture him.  (Refer to Figure~\ref{fig:G2n_1}.)  This proves the claim, and by symmetry it follows that if $m_{k-1} \not = m_k$, then $m_k \ge n-4$.

If instead $m_1 = m_2$, then the same robber strategy given in the previous paragraph shows that $m_1 \ge 4$; symmetrically, if $m_{k-1} = m_k$, then $m_k \le n-5$.

Next, we claim that for all $i \in \{1,...,k\}$, if either column $m_i$ or column $m_{i+1}$ contains only one cop, then $m_{i+1} \leq m_i+9$.  Suppose $m_i = j$ and $m_{i+1} \geq j+10$ for some $i$, and suppose by symmetry that column $m_{i}$ contains only one cop.  Without loss of generality, assume $c_i$ starts on vertex $(j,0)$.  The robber can now isolate himself on $(j+3,0)$ by starting on $(j+2,0)$ and moving right to $(j+3,0)$, down to $(j+3,1)$, right to $(j+4,1)$, up to $(j+4,0)$, and left to $(j+3,0)$.  Note that $c_i$ cannot reach the robber within three moves, and after the robber's third move, $c_i$ no longer occupies the same component as the robber.  On the other hand, $c_{i+1}$ starts too far away from the robber to reach him before he has isolated himself.  (Refer to Figure~\ref{fig:G2n_2}.) 

A similar argument shows that if both columns $m_i$ and $m_{i+1}$ contain two cops, then $m_{i+1} \le m_i+10$ (and hence $m_{i+2} \le m_i+10$).

To minimize $k$ subject to the constraints established above, we may clearly take $m_1 = 3$ and $m_{i+1} = m_i+9$ for $1 \le i \le k-1$ (where $m_k$ is reduced to $n-1$ if needed).  This yields $m_k = \min\{n-1,3+9(k-1)\}$, which suffices so long as $m_k \ge n-4$, i.e. $3+9(k-1) \ge n-4$, or $9k \ge n+2$.  Thus we obtain $k \ge \frac{n+2}{9}$; since $k$ is an integer, in fact $k \ge \ceil{\frac{n+2}{9}}$, as claimed.
\end{proof}

\begin{figure}[h]
\begin{center}
\begin{tikzpicture}
[inner sep=0mm, thick,
 vertex/.style={draw=black, circle, fill=black, minimum size=0.1cm},
 cvertex/.style={draw=black, fill=blue, circle, minimum size =0.2cm},
 rvertex/.style={draw=black, fill=red, circle, minimum size =0.2cm},
 dot/.style={draw=black,fill=black, circle, minimum size=0.04cm},
 vertexlabel/.style={draw=none, fill=none, shape=rectangle, inner sep=2pt, font=\small},
 textnode/.style={draw=none, fill=none, shape=rectangle, inner sep=2pt},
 xscale=1.0,yscale=1.0]
 
% generate grid
 
\foreach \x in {0,...,5}
   \foreach \y in {0,...,1} 
		\node at (\x,\y) [vertex] {};
\foreach \x in {0,...,5}
    \draw (\x,0) -- (\x,1);
\foreach \y in {0,...,1}
    \draw (0,\y) -- (5.5,\y);		
    
% cop vertex
\node (c1) at (4,1) [cvertex] {};
\node at ($(c1)+(0,0.2)$) [anchor=south] {$c_1$};

% robber vertex
\node (r) at (0,1) [rvertex] {};
\node at ($(r)+(-0.2,0.2)$) [anchor=east] {$r$};

% robber's path
\node (r1) at (1,1) {};
\node (r2) at (1,0) {};
\node (r3) at (0,0) {};

% draw and label arrows
\draw[->,red,ultra thick] (r) -- (r1) node[midway,vertexlabel, anchor=south] {$1$};
\draw[->,red,ultra thick] (r1) -- (r2) node[midway,vertexlabel, anchor=west] {$2$};
\draw[->,red,ultra thick] (r2) -- (r3) node[midway,vertexlabel, anchor=north] {$3$};
\draw[->,red,ultra thick] (r3) -- (r) node[midway,vertexlabel, anchor=east] {$4$};

\end{tikzpicture}
\end{center}
\caption{: Robber strategy at corner for Theorem~\ref{thm:2byn}.}
\label{fig:G2n_1}
\end{figure}
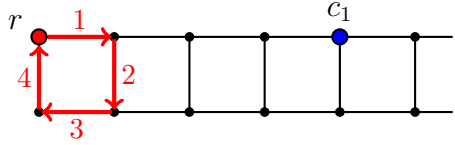

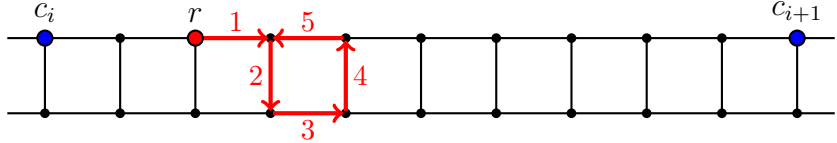
\begin{figure}[h]
\begin{center}
\begin{tikzpicture}
[inner sep=0mm, thick,
 vertex/.style={draw=black, circle, fill=black, minimum size=0.1cm},
 cvertex/.style={draw=black, fill=blue, circle, minimum size =0.2cm},
 rvertex/.style={draw=black, fill=red, circle, minimum size =0.2cm},
 dot/.style={draw=black,fill=black, circle, minimum size=0.04cm},
 vertexlabel/.style={draw=none, fill=none, shape=rectangle, inner sep=2pt, font=\small},
 textnode/.style={draw=none, fill=none, shape=rectangle, inner sep=2pt},
 xscale=1.0,yscale=1.0]
 
 % generate grid
 \foreach \x in {0,...,10}
    \foreach \y in {0,...,1} 
		    \node at (\x,\y) [vertex] {};
\foreach \x in {0,...,10}
    \draw (\x,0) -- (\x,1);
\foreach \y in {0,...,1}
    \draw (-0.5,\y) -- (10.5,\y);		
    
% cops' vertices
\node (c1) at (0,1) [cvertex] {};
\node at ($(c1)+(0,0.2)$) [anchor=south] {$c_i$};
\node (c2) at (10,1) [cvertex] {};
\node at ($(c2)+(0,0.2)$) [anchor=south] {$c_{i+1}$};

% robber vertex
\node (r) at (2,1) [rvertex] {};
\node at ($(r)+(0,0.2)$) [anchor=south] {$r$};

% robber's path
\node (r1) at (3,1) {};
\node (r2) at (3,0) {};
\node (r3) at (4,0) {};
\node (r4) at (4,1) {};

% draw and label arrows
\draw[->,red,ultra thick] (r) -- (r1) node[midway,vertexlabel, anchor=south] {$1$};
\draw[->,red,ultra thick] (r1) -- (r2) node[midway,vertexlabel, anchor=east] {$2$};
\draw[->,red,ultra thick] (r2) -- (r3) node[midway,vertexlabel, anchor=north] {$3$};
\draw[->,red,ultra thick] (r3) -- (r4) node[midway,vertexlabel, anchor=west] {$4$};
\draw[->,red,ultra thick] (r4) -- (r1) node[midway,vertexlabel, anchor=south] {$5$};

\end{tikzpicture}
\end{center}
\caption{: Robber strategy between cops for Theorem~\ref{thm:2byn}.}
\label{fig:G2n_2}
\end{figure}

Before proceeding, we remark that Lemma \ref{lem:2byn_middle} and an argument along the lines of that used for Theorem \ref{thm:2byn} together yield $\burn(P_2 \cart C_n) = \ceil{\frac{n}{9}}$ for $n \ge 10$; we omit the details. 

We next tackle general $m \times n$ grids.  As a first step toward this goal, we will actually consider $m \times n$ tori, since the analysis is simpler and uses many of the same techniques we will use for grids.  We begin by building up sufficient conditions for the cops to win on $T_{m,n}$.  Our first lemma actually applies to any graph in which all vertices have even degree, so it may be useful for graphs other than $T_{m,n}$.

\begin{lemma}\label{lem:torus_starting_vx} 
Let $G$ be a graph in which every vertex has even degree.  If at any point any cop can reach the robber's starting vertex, then the cops can capture the robber.
\end{lemma}

\begin{proof} 
Consider the game played on $G$, and suppose some cop $c$ reaches the robber's starting vertex $v$.  For the remainder of the game, the cop plays as follows:

\begin{itemize}
\item If the robber does not move closer to $v$, then the cop moves closer to the robber.
\item If the robber moves closer to $v$, then the cop moves closer to $v$.
\end{itemize}

Note that the cop ensures that she is never further from $v$ than the robber, so the robber can never safely return to $v$.  We now show that this strategy does in fact enable the cop to capture the robber.

The cop's strategy ensures that the robber cannot remain on the same vertex indefinitely, so long as there is some path joining the cop and the robber.  Thus, the only way for the robber to avoid capture is to disconnect $G$ and wind up in a component that contains no cops.  We claim that the cop guarding $v$ is always in the same component as the robber, hence the robber cannot perpetually escape capture.

The cop's strategy ensures that there is always a path from her current position to $v$, so it suffices to show that the robber is always in the component containing $v$.  Each vertex in $G$ initially has even degree, and when the robber passes through a vertex, he deletes two edges incident to that vertex.  Thus, at all points in the game, every vertex has even degree except perhaps for $v$ and the robber's current position, $u$.  Since the cop's strategy prevents the robber from returning to $v$, we must have $u \neq v$.  Thus, the graph has exactly two vertices of odd degree, namely $u$ and $v$.  Since every component of the graph must contain an even number of vertices with odd degree, $u$ and $v$ must belong to the same component, as claimed.
\end{proof}

We next apply Lemma~\ref{lem:torus_starting_vx} to establish a simpler sufficient condition for the cops to win on $T_{m,n}$.

\begin{lemma}\label{lem:torus_helper}
For the torus $T_{m,n}$ with $m \ge n$, if some cop starts the game within distance 5 of the robber, then that cop can capture the robber.
\end{lemma}
\begin{proof}
Suppose cop $c$ starts the game within distance 5 of the robber.  By symmetry we may assume that $c$ begins at vertex $(0,0)$, while the robber begins at some vertex $(i,j)$ with $0 \le j \le i$ and $i+j \le 5$.  We consider five cases (see Figure \ref{fig:torus_helper}).  In the cases below, it will be helpful to note that if $c$ can reach some neighbor of the robber's starting vertex (with the edge between the two vertices still intact), then she can either capture the robber or reach his starting vertex.

\medskip
\begin{itemize}
\item {\bf Case 1:} $(i,j) \in \{(0,0), (1,0), (2,0), (3,0), (1,1), (2,1)\}$.  If $(i,j) = (0,0)$, then $c$ captures the robber immediately.  Otherwise, she moves to the right on her first move; from here it is straightforward to verify that she can either capture the robber or reach his starting vertex, regardless of the robber's strategy.

\medskip
\item {\bf Case 2:} $(i,j) = (2,2)$.  This time, $c$ moves right on her first turn and down on her second.  On her next turn she can move either right to $(2,1)$ or down to $(1,2)$; at least one of these two vertices must still be adjacent to $(2,2)$, so she can either capture the robber or reach his starting vertex.

\medskip
\item {\bf Case 3:} $(i,j) \in \{(4,0), (1,3)\}$.  In this case, $c$ moves right on her first two turns, and again she can either capture the robber or reach his starting vertex.  (This is easy to verify by inspection unless the robber began on $(4,0)$ and moved left, then down on his first two turns; in this case $c$ can reach the robber's starting vertex by moving right, up, right, and down on her next four turns.)

\medskip
\item {\bf Case 4:} $(i,j) = (3,2)$.  This time, $c$ moves right on her first two turns and down on her third.  Once again she can either capture the robber or reach his starting vertex by way of either $(3,1)$ or $(2,2)$, since the robber has not taken enough turns to delete the edges joining each of these vertices to his starting vertex.

\medskip
\item {\bf Case 5:} $(i,j) \in \{(5,0), (4,1)\}$.  This time, $c$ moves right on her first three turns and again she can capture the robber or reach his starting vertex, similarly to Case 3.
\end{itemize}
In any case, $c$ can capture the robber, as claimed.
\end{proof}

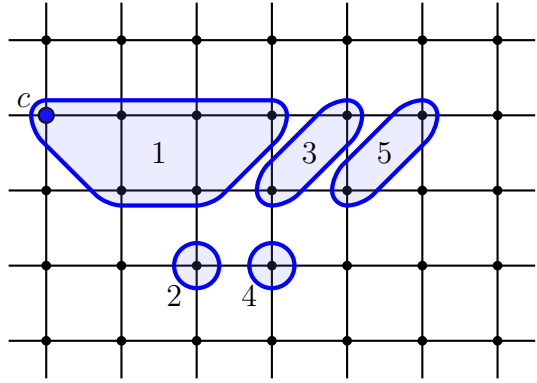
\begin{figure}[h]
\begin{center}
\begin{tikzpicture}
[inner sep=0mm, thick,
 vertex/.style={draw=black, circle, fill=black, minimum size=0.1cm},
 cvertex/.style={draw=black, fill=blue, circle, minimum size =0.2cm},
 dot/.style={draw=black,fill=black, circle, minimum size=0.04cm},
 vertexlabel/.style={draw=none, fill=none, shape=rectangle, inner sep=2pt, font=\small},
 textnode/.style={draw=none, fill=none, shape=rectangle, inner sep=2pt},
 xscale=1.0,yscale=1.0]

\foreach \x in {0,...,6}
    \foreach \y in {4,...,8} 
		\node at (\x,\y) [vertex] {};	
\foreach \x in {0,...,6}
    \draw (\x,3.5) -- (\x,8.5);
\foreach \y in {4,...,8}
    \draw (-0.5,\y) -- (6.5,\y);		
   
% cop starting vertex
\node (c1) at (0,7) [cvertex] {};
\node at ($(c1)+(-0.2,0.2)$) [anchor=east] {$c$};

% the five cases
\draw [blue, ultra thick, fill=blue!50, fill opacity=0.15, rounded corners=2mm] (-0.2,7.2) -- (-0.2,6.8) -- (0.8, 5.8) -- (2.2,5.8) -- (3.2,6.8) -- (3.2,7.2) -- cycle;
\draw [blue, ultra thick, fill=blue!50, fill opacity=0.15] (2,5) circle (0.3cm);
\draw [blue, ultra thick, fill=blue!50, fill opacity=0.15, rounded corners=2mm] (3.2,5.8) -- (2.8,5.8) -- (2.8,6.2) -- (3.8,7.2) -- (4.2,7.2) -- (4.2,6.8) -- cycle;
\draw [blue, ultra thick, fill=blue!50, fill opacity=0.15] (3,5) circle (0.3cm);
\draw [blue, ultra thick, fill=blue!50, fill opacity=0.15, rounded corners=2mm] (4.2,5.8) -- (3.8,5.8) -- (3.8,6.2) -- (4.8,7.2) -- (5.2,7.2) -- (5.2,6.8) -- cycle;

\node at (1.5,6.5) [anchor=center] {$1$};
\node at (1.7,4.6) [anchor=center] {$2$};
\node at (3.5,6.5) [anchor=center] {$3$};
\node at (2.7,4.6) [anchor=center] {$4$};
\node at (4.5,6.5) [anchor=center] {$5$};

\end{tikzpicture}
\end{center}
\caption{: Cases for Lemma~\ref{fig:torus_helper}.}
\label{fig:torus_helper}
\end{figure}

Our next sufficient condition for the cops to win is somewhat technical, but quite powerful.

\begin{lemma}\label{lem:torus_cops_corners}
In the game on $T_{m,n}$, suppose the robber begins on some vertex $(i,j)$. If some cop $c_1$ begins the game within distance 7 of $(i-1,j-1)$ and some cop $c_2$ begins the game within distance 7 of $(i+1,j+1)$, or if $c_1$ begins within distance 7 of $(i+1,j-1)$ and $c_2$ begins within distance 7 of $(i-1,j+1)$, then the cops can capture the robber.
\end{lemma}
\begin{proof}
Suppose cops $c_1$ and $c_2$ begin within distance 7 of vertices $(i-1,j-1)$ and $(i+1,j+1)$, respectively; the other case is symmetric.  If $c_1$ either begins on $(i-1,j-1)$ itself or begins both strictly to the right of and strictly below $(i-1,j-1)$, then she is within distance 5 of $(i,j)$ and can capture the robber by Lemma~\ref{lem:torus_helper}.  Thus we may suppose that $c_1$ begins either strictly to the left of or strictly above $(i-1,j-1)$; by symmetry we may assume the former.  Likewise, we may assume that $c_2$ begins strictly to the right of $(i+1,j+1)$; a similar argument suffices if she begins strictly below.  

By Lemma~\ref{lem:torus_starting_vx}, it suffices to show that at least one cop can reach $(i,j)$.  To this end, cop $c_1$ first attempts to travel either up or down to row $j-1$, then right to $(i-1,j-1)$.  Once she has reached $(i-1,j-1)$, she attempts to reach $(i,j)$ via either $(i-1,j)$ or $(i,j-1)$, provided that one of these paths remains intact.  Similarly, $c_2$ attempts to travel left or right to column $i+1$, then up to $(i+1,j+1)$, and from there to $(i,j)$ via either $(i+1,j)$ or $(i,j+1)$.  

We claim that the robber cannot prevent both $c_1$ and $c_2$ from reaching $(i,j)$.  There are two ways for the robber to thwart $c_1$: he could prevent $c_1$ from reaching $(i-1,j-1)$, or he could allow $c_1$ to reach $(i-1,j-1)$ but prevent her from then reaching $(i,j)$.  Preventing $c_1$ from reaching $(i-1,j-1)$ would require visiting some vertex $(i',j-1)$ with $i' \le i-2$.  This would take the robber at least three turns and leave him at least three steps left and up from $v$.  Since the robber takes at most six turns before $c_2$ reaches $(i+1,j+1)$, he cannot prevent $c_2$ from reaching $(i+1,j+1)$ with both length-2 paths to $(i,j)$ intact.  Consequently, $c_2$ can either reach $(i,j)$ or, if the robber attempts to traverse one of these paths, capture him directly.  Thus the robber cannot safely prevent $c_1$ from reaching $(i-1,j-1)$, nor (by symmetry) can he prevent $c_2$ from reaching $(i+1,j+1)$.  

By the time $c_1$ and $c_2$ reach $(i-1,j-1)$ and $(i+1,j+1)$ respectively, the robber has taken at most six turns, hence at least one of the two cops can reach $(i,j)$ in two steps; at this point the robber cannot prevent that cop from reaching $(i,j)$, since burning an edge along the relevant path would leave the robber on or adjacent to the cop, resulting in his capture.
\end{proof}

We will also need conditions that guarantee a robber win.  We begin with a useful lemma that applies not only to tori and grids, but to all graphs.

\begin{lemma}\label{lem:robber_distance}
Fix a graph $G$ and positive integers $k$ and $d$, and let $v$ be a vertex of $G$.  Let $d_i$ denote the robber's distance from $v$ after his $i$th move in the original graph $G$ (that is, disregarding any edge deletions that may occur during the game).  If the robber can play so that $i + d_i < d$ for $0 \le i \le k-1$, and if no cop begins within distance $d$ of $v$, then the cops cannot capture the robber before his $k$th move.
\end{lemma}
\begin{proof}
Consider an arbitrary cop $c$; it suffices to show that $c$ cannot capture the robber before his $k$th move.  Let $u$ denote $c$'s starting vertex, and for $i \in \{0, 1, \dots, k-1\}$, let $v_i$ denote the robber's position after his $i$th move.  In order for $c$ to capture the robber on $v_i$, we must have $\dist_G(u,v_i) \le i+1$, since $c$ can take at most $i+1$ steps before the robber leaves $v_i$.  However, 
$$\dist_G(u,v) \le \dist_G(u,v_i) + \dist_G(v_i,v),$$
and so
$$\dist_G(u,v_i) \ge \dist_G(u,v) - \dist_G(v_i,v) \ge (d+1) - d_i > (d+1) - (d-i) = i+1,$$
so $c$ cannot capture the robber on $v_i$.  It follows that $c$ cannot capture the robber before his $k$th move.
\end{proof}

Lemma~\ref{lem:robber_distance} leads to a useful sufficient condition for the robber to win on $G_{m,n}$ or $T_{m,n}$.

\begin{lemma}\label{lem:grid_rinside}
Fix positive integers $m$ and $n$, and consider the game on either $T_{m,n}$ or $G_{m,n}$.  If there is some vertex $v$ of degree 4 such that no cop starts within distance 5 of $v$ and at most one cop starts within distance 9, then the robber can win.
\end{lemma}
\begin{proof}
Suppose no cop starts within distance 5 of $(i,j)$ and at most one cop starts within distance 9.  If in fact no cops start within distance 9, then the robber can win by starting on $(i,j)$, then moving right to $(i+1,j)$, up to $(i+1,j-1)$, left to $(i,j-1)$, down to $(i,j)$, left to $(i-1,j)$, down to $(i-1,j+1)$, right to $(i,j+1)$, and finally up to $(i,j)$.  (Refer to Figure~\ref{fig:grid_rinside}.)  Let $d_k$ denote the distance from the robber to $(i,j)$ after the robber's $k$th turn.  The robber's strategy ensures that for $1 \le k \le 7$, we have $k + d_k < 9$.  Since no cop started the game within distance 9 of $(i,j)$, Lemma~\ref{lem:robber_distance} shows that no cop can reach the robber before his 8th move, at which point he isolates himself.

Suppose now that no cops begin within distance 5 of $(i,j)$ and exactly one cop, $c$, begins within distance 9.  Suppose without loss of generality that $c$ begins at vertex $(k,\ell)$, where $k \le i$ and $\ell \le j$.  The robber now plays as follows.  He begins at $(i,j)$, then moves up to $(i,j-1)$, left to $(i-1,j-1)$, down to $(i-1,j)$, and right to $(i,j)$.  At this point, he pauses to assess the situation.   

Since the initial distance from $c$ to $(i,j)$ was at least 6, but $c$ has taken only five turns, she cannot have yet captured the robber.  Moreover, by the assumptions that $k \le i$ and $\ell \le j$, she cannot reach $(i+1,j+1)$ in fewer than three steps, and she cannot reach either $(i+1,j)$ or $(i,j+1)$ in two steps.  We claim that $c$ cannot be within three steps of both $(i+1,j)$ and $(i,j+1)$ simultaneously.  Suppose otherwise.  For the cop to be within three steps of both $(i+1,j)$ and $(i,j+1)$ as well as at least three steps for $(i+1,j+1)$, she must occupy a vertex that is within distance 3 of both $(i+1,j)$ and $(i,j+1)$ in the original graph (before any edge deletions) and within distance 4 of $(i+1,j+1)$.  There are only five such vertices: $(i-1,j)$, $(i,j+1)$, $(i-2,j)$, $(i-1,j-1)$, and $(i,j-2)$.  If $c$ occupies $(i-1,j)$ or $(i-2,j)$ then she cannot reach $(i+1,j)$ within three steps, since the robber has deleted the edge from $(i-1,j)$ to $(i,j)$; likewise, if she occupies $(i,j-1)$ or $(i,j-2)$, then she cannot reach $(i,j+1)$ within three steps, and if she occupies $(i-1,j-1)$ then she cannot reach either vertex within three steps.

Thus, suppose $c$ cannot reach $(i,j+1)$ within three steps (the other case is similar).  The robber now moves right to $(i+1,j)$, down to $(i+1,j+1)$, left to $(i,j+1)$, and finally up to $(i,j)$.  The cop clearly cannot capture the robber on $(i+1,j)$.  She cannot capture the robber on $(i+1,j+1)$, because she takes only two moves before the robber leaves that vertex.  Likewise, she cannot capture the robber on $(i,j+1)$ because she takes only three steps before the robber returns to $(i,j)$.  Hence the robber safely returns to $(i,j)$ and, having done so, isolates himself on $(i,j)$.

Thus $c$ cannot capture the robber.  Moreover, since no other cop begins within distance 9 of the robber, an argument similar to that used at the beginning of the proof shows that no other cop can capture the robber either.  Thus, the robber wins.
\end{proof}

% \begin{lemma}\label{lem:grid_rinside} 
% Fix $m, n \ge 3$. 
% \begin{enumerate}
% \item [(1)] When playing the game on $T_{m,n}$, if there is some vertex $(i,j)$ such that no cop starts within distance 8 of $(i,j)$, then the robber can win.
% \item [(2)] When playing the game on $G_{m,n}$, if there is some vertex $(i,j)$ of degree 4 such that no cop starts within distance 8 of $(i,j)$ and, moreover, $1 \le i \le n-2$ and $1 \le j \le m-2$, then the robber can win.
% \end{enumerate}
% \end{lemma}
% \begin{proof} 
% For (1), suppose no cop starts within distance 8 of some vertex $(i,j)$ on $T_{m,n}$.  The robber can isolate himself by starting on $(i-1,j)$, then moving up to $(i-2,j)$, right to $(i-2,j+1)$, down to $(i-1,j+1)$, left to $(i-1,j)$, left to $(i-1,j-1)$, down to $(i,j-1)$, right to $(i,j)$, and finally up to $(i-1,j)$.  (Refer to Figure~\ref{fig:grid_rinside}.)  Let $d_k$ denote the distance from the robber to $(i,j)$ after the robber's $k$th turn.  The robber's strategy ensures that for $1 \le k \le 7$, we have $k + d_k < 8$.  Since no cop started the game within distance 8 of $(i,j)$, Lemma~\ref{lem:robber_distance} shows that no cop can reach the robber before his 8th move, at which point he isolates himself.

% The proof of (2) is similar.  The strategy given above works unless $j = m-2$, in which case the robber can employ a symmetric strategy.
% \end{proof}

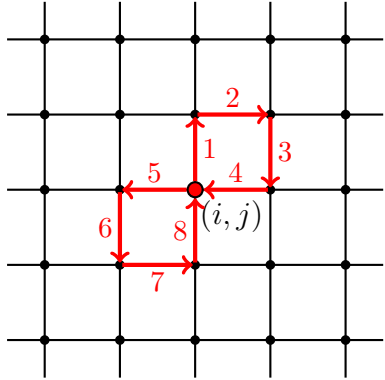
\begin{figure}[h]
\begin{center}
\begin{tikzpicture}
[inner sep=0mm, thick,
 vertex/.style={draw=black, circle, fill=black, minimum size=0.1cm},
 bigvertex/.style={draw=black, fill=black, circle, minimum size=0.2cm},
 cvertex/.style={draw=black, fill=blue, circle, minimum size =0.2cm},
 rvertex/.style={draw=black, fill=red, circle, minimum size =0.2cm},
 dot/.style={draw=black,fill=black, circle, minimum size=0.04cm},
 vertexlabel/.style={draw=none, fill=none, shape=rectangle, inner sep=2pt, font=\small},
 textnode/.style={draw=none, fill=none, shape=rectangle, inner sep=2pt},
 xscale=1.0,yscale=1.0]
 
 % generate grid
 
 \foreach \x in {0,...,4}
    \foreach \y in {0,...,4} 
		    \node at (\x,\y) [vertex] {};

\foreach \x in {0,...,4}
    \draw (\x,-0.5) -- (\x,4.5);

\foreach \y in {0,...,4}
    \draw (-0.5,\y) -- (4.5,\y);

% robber vertex
    
\node (r) at (2,2) [rvertex] {};
%\node at ($(r)+(-0.1,0.2)$) [anchor=east] {$r$};
\node at ($(r)+(0.05,-0.12)$) [anchor=north west]{$(i,j)$};

% robber's path

\node (r1) at (2,3) {};
\node (r2) at (3,3) {};
\node (r3) at (3,2) {};
\node (r4) at (2,1) {};
\node (r5) at (1,1) {};
\node (r6) at (1,2) {};

% draw and label arrows

\draw[->,red,ultra thick] (r) -- (r1) node[midway, vertexlabel, anchor=west] {$1$};
\draw[->,red,ultra thick] (r1) -- (r2) node[midway, vertexlabel, anchor=south] {$2$};
\draw[->,red,ultra thick] (r2) -- (r3) node[midway, vertexlabel, anchor=west] {$3$};
\draw[->,red,ultra thick] (r3) -- (r) node[midway, vertexlabel, anchor=south] {$4$};
\draw[->,red,ultra thick] (r) -- (r6) node[midway, vertexlabel, anchor=south] {$5$};
\draw[->,red,ultra thick] (r6) -- (r5) node[midway, vertexlabel, anchor=east] {$6$};
\draw[->,red,ultra thick] (r5) -- (r4) node[midway, vertexlabel, anchor=north] {$7$};
\draw[->,red,ultra thick] (r4) -- (r) node[midway, vertexlabel, anchor=east] {$8$};

\end{tikzpicture}
\end{center}
\caption{: Robber strategy for Lemma~\ref{lem:grid_rinside}.}
\label{fig:grid_rinside}
\end{figure}

We are finally ready to establish bounds on $\burn(T_{m,n})$.

\begin{theorem}\label{thm:torus} 
For all positive integers $m$ and $n$, 
$$\ceil{\frac{mn}{121}} \leq \burn(T_{m,n}) \leq 2\ceil{\frac{m}{16}} \ceil{\frac{n}{14}}.$$
\end{theorem}
\begin{proof}
We begin with the lower bound.  Consider an initial cop placement on $T_{m,n}$.  We define a weighting function on $V(T_{m,n})$ as follows: a vertex having $k$ cops within distance 5 and another $\ell$ within distance 9 receives weight $k+\ell/2$.  If any vertex has weight less than 1, then by Lemma~\ref{lem:grid_rinside}, the robber has a winning strategy.  Thus, for the cops to win, every vertex must have weight at least 1, so the sum of the weights of all vertices in $T_{m,n}$ must be at least $mn$.  For any vertex $v$ in $T_{m,n}$, there are at most 61 vertices within distance 5 of $v$ and at most an additional 120 within distance 9; consequently, each cop's contribution to the total weight is at most $61 + 120/2$, or $121$.  Thus the total number of cops must be at least $mn/121$, hence $\burn(T_{m,n}) \ge \ceil{\frac{mn}{121}}$.

For the upper bound, we give an initial cop placement and claim that regardless of where the robber starts, at least one cop can either capture the robber or reach the robber's starting vertex (at which point she can capture the robber using the strategy outlined in Lemma \ref{lem:torus_starting_vx}).  For all $0 \le k < 2\ceil{n/16}$ and $0 \le \ell < 2\ceil{m/16}$ such that $k+\ell$ is odd, we place a cop at $(7k,8\ell)$.  (Throughout the proof, for any vertex $(i,j)$, we take $i$ modulo $n$ and $j$ modulo $m$ as needed.)

Now consider the $9 \times 8$ block of vertices with upper-left corner $(7k,8\ell)$ for some $k, \ell$ such that $0 \le 7k < n$ and $0 \le 8\ell < m$.  Suppose $k+\ell$ is even (the case where $k+\ell$ is odd is symmetric).  Cops occupy the lower-left and upper-right vertices of this block, namely $(7k,8\ell+8)$ and $(7k+7,8\ell)$; denote these cops by $c_1$ and $c_2$, respectively.  (Refer to Figure~\ref{fig:torus_cases}.)

We claim that if the robber begins anywhere within this block, then the cops can capture him.  Indeed, every vertex in this block satisfies the hypotheses of Lemma~\ref{lem:torus_helper} or Lemma~\ref{lem:torus_cops_corners}, except perhaps for the top-left and lower-right corners, i.e. $(7k,8\ell)$ and $(7k+7,8\ell+8)$.  To show that the cops can capture robbers who begin at these vertices, we consider several cases.

\begin{itemize}
\item {\bf Case 1:} the robber starts at $(7k,8\ell)$ with $k \ge 1$.  In this case, there is also a cop $c_3$ at vertex $(7k-7,8\ell)$.  Now cops $c_2$ and $c_3$ can capture the robber by Lemma~\ref{lem:torus_cops_corners}.

\medskip
\item {\bf Case 2:} the robber starts at $(7k,8\ell)$ with $k = 0$, i.e. at $(0,8\ell)$.  Unlike in case 1, this time there need not be a cop at $(-7,8\ell)$.  The cops' strategy ensures that in row $8\ell$, cops appear 14 columns apart.  In particular, there is a cop $c_3$ at vertex $(i,8\ell)$ for some $i \in \{-7, -6, \dots, 6\}$.  If $i \in \{-7,-6\}$, then cops $c_2$ and $c_3$ can capture the robber by Lemma~\ref{lem:torus_cops_corners}.  If instead $i \in \{-5, -4, \dots, 5\}$, then the robber starts within distance 5 of $c_3$, who can thus capture him by Lemma~\ref{lem:torus_helper}.  Finally, if $i = 6$, then $c_3$ occupies vertex $(6,8\ell)$, while some other cop $c_4$ occupies $(-1,8\ell+8)$.  Now $c_3$ and $c_4$ can capture the robber by Lemma~\ref{lem:torus_cops_corners}.  

\medskip
\item {\bf Case 3:} the robber starts at $(7k+7,8\ell+8)$.  The cops' placement ensures that there are cops at $(7k+14,8\ell+8)$ and $(7k+7,8\ell+16)$, so the cops can capture the robber as in Case 1.
\end{itemize}
Thus the cops can capture the robber if he starts anywhere within the $9 \times 8$ block.  Since every vertex in the torus belongs to at least one such block, the cops can always capture the robber.

\begin{figure}[h]
\begin{center}
\begin{tikzpicture}
[inner sep=0mm, thick,
 vertex/.style={draw=black, circle, fill=black, minimum size=0.1cm},
 cvertex/.style={draw=black, ultra thick, fill=blue, circle, minimum size =0.2cm},
 dot/.style={draw=black,fill=black, circle, minimum size=0.04cm},
 vertexlabel/.style={draw=none, fill=none, shape=rectangle, inner sep=2pt, font=\small},
 textnode/.style={draw=none, fill=none, shape=rectangle, inner sep=2pt},
 xscale=1.0,yscale=1.0]

% grid
\foreach \x in {0,...,7}
    \foreach \y in {0,...,8} 
		    \node at (\x,\y) [vertex] {};
\foreach \x in {0,...,6}
    \draw (\x,-0.7) -- (\x,8.7);
\draw (7,-0.15) -- (7,8.7);
\foreach \y in {0,...,8}
    \draw (-0.7,\y) -- (7.7,\y);		

% label a couple special vertices
\node at (6.2,-0.3) [anchor=north west] {$(7k+7,8\ell+8)$};
\node at (-0.2,8.2) [anchor=south east] {$(7k,8\ell)$};

% cop vertices
\node (c1) at (0,0) [cvertex] {};
\node at ($(c1)+(-0.18,-0.18)$) [anchor=north east] {$c_1$};
\node (c2) at (7,8) [cvertex] {};
\node at ($(c2)+(0.18,0.18)$) [anchor=south west] {$c_2$};

% vertices within distance 5 of a cop
\draw[blue, ultra thick, fill=blue!50, fill opacity=0.15, rounded corners] ($(c1)+(-0.2,-0.2)$) -- ($(c1)+(-0.2,5.4)$) -- ($(c1)+(5.4,-0.2)$) -- cycle;
\draw[blue, ultra thick, fill=blue!50, fill opacity=0.15, rounded corners] ($(c2)+(0.2,0.2)$) -- ($(c2)+(0.2,-5.4)$) -- ($(c2)+(-5.4,0.2)$) -- cycle;

% vertices with corners within distance 6
\draw[purple, ultra thick, fill=purple!50, fill opacity=0.15, rounded corners] ($(c1)+(5.7,-0.2)$) -- ($(c1)+(7.2,-0.2)+(180:0.9)$) arc (180:90:0.9) -- ($(c2)+(0.2,-5.7)$) -- ($(c2)+(-5.7,0.2)$) -- ($(c1)+(-0.2,8.2)+(360:0.9)$) arc (360:270:0.9) -- ($(c1)+(-0.2,5.7)$) -- cycle;

% the two remaining vertices
\draw[red, ultra thick, fill=red!50, fill opacity=0.15, rounded corners] ($(c1)+(7.2,-0.2)+(180:0.7)$) arc (180:90:0.7) -- ($(c1)+(7.2,-0.2)$) -- cycle;
\draw[red, ultra thick, fill=red!50, fill opacity=0.15, rounded corners] ($(c1)+(-0.2,8.2)+(270:0.7)$) arc (270:360:0.7) -- ($(c1)+(-0.2,8.2)$) -- cycle;

\end{tikzpicture}
\end{center}
\caption{: the $9 \times 8$ block with upper-left corner $(7k,8\ell)$, for $k+\ell$ even.  Vertices within the two triangles satisfy the hypotheses of Lemma~\ref{lem:torus_helper}; vertices within the central region satisfy the hypotheses of Lemma~\ref{lem:torus_cops_corners}; the remaining two vertices satisfy neither.}
\label{fig:torus_cases}
\end{figure}
\end{proof}

Note that the lower bound on $\burn(T_{m,n})$ in Theorem~\ref{thm:torus} is about $mn/121$, while the upper bound is about $mn/112$.  With a somewhat more detailed argument, the lower bound can be improved to $mn/120$.  However, we suspect that in fact $\burn(T_{m,n}) \sim mn/112$, i.e. that the upper bound is asymptotically tight.

On the grid, the situation is slightly more complex.  The edges and corners of the grid contain vertices of low degree, which may allow the robber to isolate himself more quickly than he could in the middle of the grid.  Thus when playing on the grid, we must use more cops than when playing on the torus.

\begin{lemma}\label{lem:grid_robber} 
In the game on $G_{m,n}$, if there exists some vertex $v$ of degree 2 or 3 such that no cop starts within distance 5 of $v$, then the robber can win.  
\end{lemma}
\begin{proof} 
Suppose no cop starts within distance 5 of some vertex $v$ having degree 2 or 3.  By symmetry, we may suppose $v = (i,0)$ for some $i \in \{0, 1, \dots, n-1\}$.  We consider two cases.

\begin{itemize}
\item {\bf Case 1:} $i \le 1$.  In this case, the robber can isolate himself on $(0,0)$ by starting on $(0,0)$ and moving down to $(0,1)$, right to $(1,1)$, up to $(1,0)$, and left to $(0,0)$.  (Refer to Figure~\ref{fig:grid_rcorner}.)  Let $d_k$ denote the robber's distance from $(1,0)$ after $k$ moves.  The robber's strategy ensures that $k + d_k < 4$ for $0 \le k \le 3$.  Since no cop starts within distance 5 of $v$ it follows that no cop starts within distance 4 of $(1,0)$, so by Lemma~\ref{lem:robber_distance}, no cop can reach the robber before his fourth turn.  Hence the robber successfully isolates himself and thus wins.

\medskip
\item {\bf Case 2:} $i \ge 2$.  In this case, the robber can isolate himself on $(i-1,0)$ by starting on $(i-2,0)$ and moving right to $(i-1,0)$, down to $(i-1,1)$, right to $(i,1)$, up to $(i,0)$, and left to $(i-1,0)$.  (Refer to Figure~\ref{fig:grid_rborder}.)  Letting $d_k$ denote the robber's distance from $v$ after $k$ moves, the robber's strategy ensures that $k + d_k < 5$ for $0 \le k \le 4$, so by Lemma~\ref{lem:robber_distance}, no cop can reach the robber before he has isolated himself.
\end{itemize}
\end{proof}

\begin{figure}[h]
\begin{center}
\begin{tikzpicture}
[inner sep=0mm, thick,
 vertex/.style={draw=black, circle, fill=black, minimum size=0.1cm},
 cvertex/.style={draw=black, fill=blue, circle, minimum size =0.2cm},
 rvertex/.style={draw=black, fill=red, circle, minimum size =0.2cm},
 dot/.style={draw=black,fill=black, circle, minimum size=0.04cm},
 vertexlabel/.style={draw=none, fill=none, shape=rectangle, inner sep=2pt, font=\small},
 textnode/.style={draw=none, fill=none, shape=rectangle, inner sep=2pt},
 xscale=1.0,yscale=1.0]
 
 % generate grid
 
 \foreach \x in {0,...,2}
    \foreach \y in {0,...,2} 
		    \node at (\x,\y) [vertex] {};

\foreach \x in {0,...,2}
    \draw (\x,-0.5) -- (\x,2);

\foreach \y in {0,...,2}
    \draw (0,\y) -- (2.5,\y);

% robber vertex
    
\node (r) at (0,2) [rvertex] {};
\node at ($(r)+(-0.1,0.2)$) [anchor=east] {$v$};

% robber's path

\node (r1) at (0,1) {};
\node (r2) at (1,1) {};
\node (r3) at (1,2) {};

% draw and label arrows

\draw[->,red,ultra thick] (r) -- (r1) node[midway,vertexlabel, anchor=east] {$1$};
\draw[->,red,ultra thick] (r1) -- (r2) node[midway,vertexlabel, anchor=north] {$2$};
\draw[->,red,ultra thick] (r2) -- (r3) node[midway,vertexlabel, anchor=west] {$3$};
\draw[->,red,ultra thick] (r3) -- (r) node[midway,vertexlabel, anchor=south] {$4$};

\end{tikzpicture}
\end{center}
\caption{: Robber strategy for Case 1 of Lemma~\ref{lem:grid_robber}.}
\label{fig:grid_rcorner}
\end{figure}

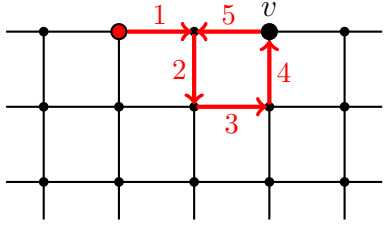
\begin{figure}[h]
\begin{center}
\begin{tikzpicture}
[inner sep=0mm, thick,
 vertex/.style={draw=black, circle, fill=black, minimum size=0.1cm},
 bigvertex/.style={draw=black, circle, fill=black, minimum size=0.2cm},
 cvertex/.style={draw=black, fill=blue, circle, minimum size =0.2cm},
 rvertex/.style={draw=black, fill=red, circle, minimum size =0.2cm},
 dot/.style={draw=black,fill=black, circle, minimum size=0.04cm},
 vertexlabel/.style={draw=none, fill=none, shape=rectangle, inner sep=2pt, font=\small},
 textnode/.style={draw=none, fill=none, shape=rectangle, inner sep=2pt},
 xscale=1.0,yscale=1.0]
 
 % generate grid
 
 \foreach \x in {0,...,4}
    \foreach \y in {0,...,2} 
		    \node at (\x,\y) [vertex] {};

\foreach \x in {0,...,4}
    \draw (\x,-0.5) -- (\x,2);

\foreach \y in {0,...,2}
    \draw (-0.5,\y) -- (4.5,\y);

% robber vertex
    
\node (r) at (1,2) [rvertex] {};
\node (v) at (3,2) [bigvertex] {};
\node at ($(v)+(0,0.2)$) [anchor=south] {$v$};

% robber's path

\node (r1) at (2,2) {};
\node (r2) at (2,1) {};
\node (r3) at (3,1) {};
\node (r4) at (3,2) {};

% draw and label arrows

\draw[->,red,ultra thick] (r) -- (r1) node[midway,vertexlabel, anchor=south] {$1$};
\draw[->,red,ultra thick] (r1) -- (r2) node[midway,vertexlabel, anchor=east] {$2$};
\draw[->,red,ultra thick] (r2) -- (r3) node[midway,vertexlabel, anchor=north] {$3$};
\draw[->,red,ultra thick] (r3) -- (v) node[midway,vertexlabel, anchor=west] {$4$};
\draw[->,red,ultra thick] (v) -- (r1) node[midway,vertexlabel, anchor=south] {$5$};

\end{tikzpicture}
\end{center}
\caption{: Robber strategy for Case 2 of Lemma~\ref{lem:grid_robber}.}
\label{fig:grid_rborder}
\end{figure}

\begin{theorem}\label{thm:grid} 
For all positive integers $m$ and $n$,  
$$\ceil{\frac{mn}{121}} \leq \burn(G_{m,n}) \leq 2\floor{\frac{m}{16}} \floor{\frac{n}{14}} + 3 \left (\floor{\frac{m}{5}} + \floor{\frac{n}{5}} \right ) + 4.$$
\end{theorem}
\begin{proof} 
The lower bound follows by the same argument as in Theorem~\ref{thm:grid}.  For the upper bound, we give a winning cop strategy.  First, we place cops at vertices $(1,5k+2)$ and $(n-2,5k+2)$ for all $k$ such that $0 \le 5k+2 < m$, along with vertices $(1,m-1)$ and $(n-2,m-1)$.  Similarly, we place cops at vertices $(5\ell+2,1)$ and $(5\ell+2,m-2)$ for all $\ell$ such that $0 \le 5\ell+2 < n$, along with vertices $(n-1,1)$ and $(n-1,m-2)$.  These cops will prevent the robber from safely visiting the border of the grid; we refer to them as the {\em border patrol} cops.  Next, we place some {\em central cops} in a manner similar to that used in the proof of Theorem \ref{thm:torus}: for all nonnegative integers $k$ and $\ell$ such that $7k < n$, $8\ell < m$, and $k+\ell$ is odd, we place a cop at $(7k,8\ell)$.  Finally, we place {\em peripheral cops} in columns $0, 10, 20, \dots$ of row $m-8$ and in columns $5, 15, \dots$ of row $m-2$.  Similarly, we place $\floor{m/5}$ peripheral cops in rows $0, 10, 20, \dots$ of column $n-2$ and in rows $5, 15, \dots$ of column $n-8$.

Note that there are at most $2\floor{m/5} + 2\floor{n/5} + 4$ border patrol cops, at most $\frac{1}{2}\floor{m/8}\floor{n/7}$ central cops, and $\floor{m/5}+\floor{n/5}$ peripheral cops, so the total number of cops used is at most 
$$2\floor{\frac{m}{16}} \floor{\frac{n}{14}} + 3 \left (\floor{\frac{m}{5}} + \floor{\frac{n}{5}} \right ) + 4$$
as claimed.

To show that the cops can win from this starting position, we first show that the border patrol cops can prevent the robber from safely visiting the border of the grid.  In particular, we give a strategy for the border patrol cops placed on row 1; the other border patrol cops use a symmetric strategy.  Consider a border patrol cop $c$ who begins on vertex $(k,1)$.  We say that columns $k-2, k-1, \dots, k+2$ are {\em assigned} to $c$.  Throughout the bulk of the game, $c$ moves left or right within row 1 and within her assigned columns.  On her turn, if the robber is not in row 0, then $c$ moves horizontally toward the robber, except that she never moves farther left than column $k-2$ or farther right than column $k+2$.  (If $c$ is in the same column as the robber, then she remains in place.)  

Suppose the robber begins at vertex $(i,j)$. To show that $c$ prevents the robber from safely entering row 0 of her assigned columns, we consider three cases.
\begin{itemize}
\item {\bf Case 1:} $j \ge 2$.  In this case, $c$ can make at least two moves before the robber enters row 0.  With these moves she can either reach the same column as the robber or the assigned column closest to him, and henceforth it is clear that she can prevent the robber from entering row 0 in her assigned columns.  

\medskip
\item {\bf Case 2:} $j = 1$.  As before, $c$ clearly prevents the robber from entering row 0 in her assigned columns unless $i \in \{k-2, k+2\}$.  Suppose without loss of generality that $i = k+2$.  On her first turn, $c$ moves right from $(k,1)$ to $(k+1,1)$.  The robber cannot remain in place or move left, lest $c$ capture him.  He cannot move right, since column $k+3$ either does not exist or, if it does exist, is assigned to another border patrol cop, who must currently occupy $(k+4,1)$.  Thus, he must move up to $(k+2,0)$ or down to $(k+2,2)$.  In the former case, $c$ moves up to $(k+1,0)$, and the robber is trapped: he cannot move down, he cannot remain still or move left lest $c$ capture him, and he cannot move right (since either column $k+3$ does not exist or, if it does, its assigned cop prevents him from entering).  In the latter case, the cops can clearly prevent the robber from ever returning to row 1 and hence from ever reaching row 0.

\medskip
\item {\bf Case 3:} Finally, suppose $j=0$.  The robber cannot move down to row 1 on his first turn, since this would result in capture by whichever cop was assigned column $i$.  However, if he remains in row 0 on his first turn, then the cops prevent him from safely leaving row 0.  Now he cannot remain still indefinitely or he will be captured, but he cannot move indefinitely since he will eventually run out of edges.  Thus, the cops eventually capture him.
\end{itemize}
We may thus suppose that the robber never enters row 0, row $m-1$, column 0, or column $n-1$.

Next, we establish analogues of Lemmas~\ref{lem:torus_starting_vx}, \ref{lem:torus_helper}, and  \ref{lem:torus_cops_corners} for use on $G_{m,n}$.  The proof of Lemma~\ref{lem:torus_starting_vx} does not apply on the grid because the grid has vertices of odd degree.  However, the border patrol cops prevent the robber from entering any of these vertices.  Thus these vertices continue to have odd degree throughout the game and, moreover, belong to the same component at all times.  It now follows, as in the proof of Lemma~\ref{lem:torus_starting_vx} that the robber must always occupy the same component as his starting vertex; consequently, if some cop can reach the robber's starting vertex, then she can capture him.  Thus we may henceforth apply Lemma~\ref{lem:torus_starting_vx}, from which it follows that we may also apply Lemmas~\ref{lem:torus_helper} and \ref{lem:torus_cops_corners} as well (provided in both cases that the robber begins at a vertex of degree 4, but the border patrol cops force him to do so).

To complete the proof, suppose the robber begins at vertex $(i,j)$.  If $i=1$, $j=1$, $i=n-1$, or $j=m-1$, then the border patrol cops can capture the robber as explained above.  If $2 \le i \le n-8$ and $2 \le j \le m-9$, then the central cops can capture him as in the proof of Theorem~\ref{thm:torus}.  Otherwise, some peripheral cop begins within distance 5 of the robber and so can capture him by Lemma~\ref{lem:torus_helper}.
\end{proof}

To finish this section, we determine the bridge-burning cop number of the $n$-dimensional hypercube $Q_n$ -- that is, the $n$-dimensional grid with vertices in $\{0,1\}^n$.

\begin{theorem}\label{thm:hypercube}
For all positive integers $n$, we have $\burn(Q_n) = 1$.
\end{theorem}
\begin{proof}
It suffices to give a strategy for one cop to capture the robber on $Q_n$.  As usual, we view the vertex set of $Q_n$ as $\{0,1\}^n$.  We refer to a vertex with a 1 in its $k^{\text{th}}$ coordinate as a vertex in the {\em $k^{th}$ dimension}.  Additionally, we say that the cop or robber {\it visits} the $k^{\text{th}}$ dimension by starting the game on or moving to a vertex in that dimension, i.e. by changing the $k^{\text{th}}$ coordinate of their position from 0 to 1.  Similarly, we will say that the cop or robber {\it leaves} the $k^{\text{th}}$ dimension by moving from a vertex inside the dimension to one outside, i.e. by changing the $k^{\text{th}}$ coordinate of their position from 1 to 0.

The cop begins at vertex $(1,1,\dots,1)$.  We claim that if, after some cop turn, there exists some $k$ such that the players' positions differ only in coordinate $k$ and the robber hasn't yet visited the $k^{\text{th}}$ dimension, then the cop has a winning strategy.

Indeed, if this situation arises, then all edges incident to vertices in dimension $k$ must be present since the robber hasn't visited the $k^{\text{th}}$ dimension.  In addition, the cop must be in the $k^{\text{th}}$ dimension, since the players' positions differ in coordinate $k$.  From this point on, the cop always mirrors the robber's move: that is, if the robber changes his $j^{\text{th}}$ coordinate, then the cop changes hers to match.  (Of course, if the robber changes his $k^{\text{th}}$ coordinate, then he has moved onto the cop's vertex and thus loses; if the robber remains on his current vertex, then the cop moves onto the robber's vertex and wins, which is possible since all edges incident to vertices in the $k^{\text{th}}$ dimension are still present.)  The cop can always do this, since all edges joining vertices in the $k^{\text{th}}$ dimension are still present.  Moreover, the cop's strategy prevents the robber from ever visiting the $k^{\text{th}}$ dimension and thus ensures that after all cop turns, the two players' positions agree in all coordinates except the $k^{\text{th}}$.  Since the graph is finite, the robber cannot keep moving forever, so the cop eventually captures him.

Hence, we need only show that the cop can always reach a vertex adjacent to the robber and in a dimension that the robber has not yet visited.  On her first turn, the cop moves closer to the robber in any coordinate she wishes.  Henceforth, the cop plays as follows.  If the robber moves away from the cop by changing his $j$th coordinate, then the cop changes her $j$th coordinate in the same way; if the robber sits still or moves closer to the cop, then the cop takes one step closer to the robber (in any direction she wishes).  As long as there is some dimension $k$ that the cop occupies and the robber has never visited, all edges in dimension $k$ are intact and hence the cop can always employ this strategy.  It suffices to show that some such dimension always exists up until the point where the cop reaches a vertex adjacent to the robber.

Suppose for the sake of contradiction that the robber can visit all $n$ dimensions while avoiding capture.  The cop's strategy ensures that once the players' positions agree in some coordinate, they will continue to do so after every cop turn so long as the robber has not visited all $n$ coordinates.  Hence once the robber has visited $n-1$ different dimensions, the players' positions agree in $n-1$ coordinates, meaning that the cop is adjacent to the robber -- from which it follows that the cop eventually wins.
\end{proof}

We remark that the argument used to prove Theorem~\ref{thm:hypercube} can be applied more generally to graphs of the form $G_1 \cart G_2 \cart \dots \cart G_n$, where each $G_i$ is one of $P_2$, $P_3$, and $C_3$.  The details are nearly identical to those given above and have been omitted.\\
\end{section}

% capture time
\begin{section}{Capture Time}\label{sec:capt}

In this section, we look not at the cop number, but at a related concept.  In the usual model of Cops and Robbers, the {\em capture time} of a graph $G$, denoted $\capt(G)$, is a measure of how quickly the cops can capture the robber.  Formally, $\capt(G)$ is the minimum number of rounds needed for the cops to guarantee a win, provided that there are exactly $c(G)$ cops.  Capture time was introduced for by Bonato et al.~\cite{BGHK09}, who showed that $\capt(G) \le n-3$ whenever $G$ has cop number 1; this bound was later improved by Gaven\v{c}iak~\cite{Gav10} to $\capt(G) \le n-4$ under the additional condition that $G$ has at least 7 vertices.

We denote capture time in the bridge-burning model by $\captb(G)$, and we aim to determine the maximum capture time of an $n$-vertex graph on which a single cop can win.  We start with an easy upper bound.
    
\begin{theorem}\label{thm:capt_upper} 
For any graph $G$ where one cop can capture the robber, $capt(G) = O(n^3)$. 
\end{theorem}

\begin{proof}
In a game on $G$, the robber can move at most $\size{E(G)}$ times.  Between moves, the robber can remain on his current vertex no more than $n$ times provided that the cop is playing optimally, since the cop will move on each turn and will never revisit a vertex while the robber remains in place.  Thus, the number of rounds needed for the cop to win is at most $\size{E(G)} \cdot n$, which is $O(n^3)$.
\end{proof}

One might expect capture times to be lower, in general, in the bridge-burning model of Cops and Robbers than in the ordinary model, since the graph necessarily becomes smaller as the game proceeds.  However, for a lower bound on the maximum capture time in the bridge-burning model, we give a graph $G$ with $\captb(G) = \Omega(n^2)$ -- an order of magnitude larger than the maximum capture time under the usual model!

\begin{theorem}\label{thm:capt_lower} 
There exists a graph $G$ such that $capt(G) = \Omega(n^2)$. 
\end{theorem}

\begin{proof} 
Let $k,m$ be positive integers with $m(k-1)$ even.  Consider the graph $G_{m,k}$ formed as follows.  Begin with a complete graph with vertices $v_1,...v_k$ and a complete $k$-partite graph with partite sets $S_1,S_2,...,S_k$, each containing exactly $m$ vertices.  For all $i \in \{1, \dots, k\}$, add an edge between $v_i$ and every vertex in $S_i$.  Finally, to each vertex $v_i$, add a pendant neighbor $u_i$.

%We first note that the cop must start on one of the vertices $v_1,...,v_k$, for if the cop starts on a vertex in $M_i$ or on $u_i$, then the robber can start on $v_j$ where $i \neq j$ and isolate himself on $u_j$ after one round and win the game.

Vertices $v_1,...,v_k$ form a dominating clique of $G$, so by Theorem \ref{thm:dominating_cliques}, a single cop can capture the robber.  To show that the cop cannot win too quickly, we give a strategy for the robber to avoid capture for a long while.  If the cop begins anywhere except on some $v_i$, then the robber begins on some $v_j$ not adjacent to the cop; on the robber's first turn, he moves to $u_j$, thereby isolating himself and winning the game.  Thus we may suppose the cop begins on some $v_i$.  

The robber starts on a vertex in $S_j$ for some $j \not = i$.  In addition, the robber fixes some Eulerian cycle in the subgraph of $G$ induced by $S_1 \cup S_2 \cup \dots \cup S_k$; he can do so because this subgraph is regular of degree $m(k-1)$, which by assumption is even. If the cop moves to $u_i$ or to some vertex in $S_i$, then the robber moves to $v_j$ and, on his next turn, to $u_j$, thereby isolating himself.  If the cop moves to $v_{\ell}$ for some $\ell \not = j$, then the robber remains on his current vertex.  Finally, suppose the cop moves to $v_j$.  The robber must move, so he moves to the next vertex in his chosen Eulerian cycle.  The robber maintains this strategy until he has completed the entire cycle, at which point he remains on his current vertex and awaits his imminent capture.
  
The robber's strategy ensures that he cannot be captured before burning all edges of the complete multipartite graph induced by $S_1 \cup S_2 \cup \dots \cup S_k$, of which there are $(mk \cdot m(k-1))/2$.  Thus the number of rounds needed for the cop to win is at least $m^2 k(k-1)/2 + 1$.  Letting $\size{V(G)} = n$, we have $n = km+2k$, so $n = mk + O(k)$ and so
$$\captb(G) \ge m^2 k(k-1)/2 + 1 = \Omega(n^2).$$
\end{proof} 

The lower bound in Theorem~\ref{thm:capt_lower} differs by an order of magnitude from the upper bound in Theorem~\ref{thm:capt_upper}; we conjecture that the upper bound gives the correct order of growth.

\begin{conj}\label{conj_capt}
There exists an $n$-vertex graph $G$ with $\burn(G) =1$ and $\captb(G) = \Omega(n^3)$.
\end{conj}
\end{section}

% future work
\begin{section}{Open Problems}\label{sec:open}

We conclude the paper by suggesting a few directions for future research on the bridge-burning game.
\begin{itemize}
\medskip
\item {\bf Characterize the graphs with bridge-burning cop number 1}.  A nice structural characterization is known for graphs with cop number 1 under the usual model (see~\cite{NW83}, \cite{Qui78}), but we have no such characterization for the bridge-burning model.  One principal difficulty in tackling this problem is that in the bridge-burning model, the graph changes as the game progresses, so any structural properties satisfied by the graph at the beginning of the game need not be satisfied throughout the game.

\medskip
\item {\bf Determine the asymptotics of $\burn(T_{m,n})$ and $\burn(G_{m,n})$}.  We have shown that both parameters are asymptotically $c \cdot mn$ for some constant $c$ between 112 and 121; could it be that both are asymptotically $mn/112$?

\medskip
\item {\bf Study the game on grids of arbitrary dimension}.  Theorem~\ref{thm:hypercube} provides a first step toward this problem, but it is not clear how or if the techniques used therein would extend to grids with larger side lengths.

\medskip
\item {\bf Examine Cartesian products of general trees and/or cycles}.  The cop numbers for products of trees and for products of cycles have been completely determined under the usual model of Cops and Robbers (see~\cite{MM87,NN93}) as well as for several variants.  It would be interesting to see how the situation differs in the bridge-burning model.\\
\end{itemize}
\end{section}

% Here's an old bibliography with a few of the standard Cops and Robbers references -- maybe some will be relevant.

\end{document}